\documentclass[12pt,english]{amsart}

\usepackage[latin9]{inputenc}
\usepackage{geometry}
\usepackage{babel}
\usepackage{amsmath}
\usepackage{amsthm}
\usepackage{amsbsy}
\usepackage{amstext}
\usepackage{amssymb}
\usepackage{esint}
\usepackage{mathrsfs}
\usepackage[all]{xy}
\usepackage[unicode=true,pdfusetitle,
 bookmarks=true,bookmarksnumbered=false,bookmarksopen=false,
 breaklinks=false,pdfborder={0 0 1},backref=false,colorlinks=false]
 {hyperref}

\usepackage{todonotes}
\makeatletter
\theoremstyle{plain}
\newtheorem{thm}{\protect\theoremname}
  \theoremstyle{definition}
  \newtheorem{defn}[thm]{\protect\definitionname}
  \theoremstyle{plain}
  
  \theoremstyle{plain}
  \newtheorem{lem}[thm]{Lemma}
  \newtheorem{prop}[thm]{\protect\propositionname}
  \theoremstyle{remark}
\newtheorem{remark}[thm]{Remark}
\newtheorem{rems}[thm]{Remarks}
\theoremstyle{definition}
\newtheorem{oep}[thm]{Other eigenvalue problems}
 \newtheorem{exa}[thm]{Example}

\usepackage{subfig}

\usepackage{tikz}
\usetikzlibrary{calc}
\newcommand\curpos{\the\tikz@lastxsaved,\the\tikz@lastysaved}

\newlength{\myArraycolsep}
\@ifundefined{showcaptionsetup}{}{%
 \PassOptionsToPackage{caption=false}{subfig}}
\usepackage{subfig}
\makeatother

  \providecommand{\corollaryname}{Corollary}
  \providecommand{\definitionname}{Definition}
  \providecommand{\propositionname}{Proposition}
\providecommand{\theoremname}{Theorem}
\newcommand{\Stek}{\operatorname{Stek}}

\newcommand{\Tr}{\operatorname{Tr}}
\newcommand{\R}{\mathbb{R}}
\newcommand{\Z}{\mathbb{Z}}
\newcommand{\orb}{\mathcal{O}}
\newcommand{\oF}{\overline{F}}
\newcommand{\hm}{\widehat{M}}
\newcommand{\bs}{\backslash}
\newcommand{\D}{\mathcal{D}}
\newcommand{\bsl}{\backslash}

\begin{document}

\title{Robin and Steklov isospectral manifolds}

\author{Carolyn Gordon, Peter Herbrich, and David Webb}
\begin{abstract}
We use two of the most fruitful methods for constructing isospectral manifolds, the Sunada method and the torus action method, to construct manifolds whose Dirichlet-to-Neumann operators are isospectral at all frequencies.   The manifolds are also isospectral for the Robin boundary value problem for all choices of Robin parameter.  As in the sloshing problem, we can also impose mixed Dirichlet-Neumann conditions on parts of the boundary.  Among the examples we exhibit are Steklov isospectral flat surfaces with boundary, planar domains with isospectral sloshing problems, and Steklov isospectral metrics on balls of any dimension greater than $5$. In particular, the latter are the first examples of Steklov isospectral manifolds of dimension greater than $2$ that have connected boundaries.  
\end{abstract}



\address{Department of Mathematics, Dartmouth College, Hanover, NH, USA}

\email{carolyn.s.gordon@dartmouth.edu}

\address{Department of Mathematics, Dartmouth College, Hanover, NH, USA}

\email{peter.herbrich@gmail.com}

\address{Department of Mathematics, Dartmouth College, Hanover, NH, USA}

\email{david.l.webb@dartmouth.edu}

\maketitle
\thispagestyle{empty}

\begin{minipage}[t]{1\columnwidth}%
\global\long\def\mate#1{#1'}
 \global\long\def\principleOrbits#1{\widehat{#1}}
 \global\long\def\restriction#1#2{#1\vert_{#2}}
 \global\long\def\conjugate#1{\overline{#1}}

\global\long\def\manifold{M}
 \global\long\def\boundaryDirichlet{D}
 \global\long\def\boundaryNeumann{N}
 \global\long\def\boundaryRobin{R}

\global\long\def\union{\sqcup}

\global\long\def\group{G}
 \global\long\def\subgroup{H}
 \global\long\def\torus{T}
 \global\long\def\subtorus{W}

\global\long\def\isometry{A}
\global\long\def\diffeo{F_{\subtorus}}
 \global\long\def\projection#1{\overline{#1}}
 \global\long\def\diffeoInduced{\projection F_{\subtorus}}
 \global\long\def\projectionMap#1{\pi_{#1}}

\global\long\def\parameterSteklov{\alpha}
 \global\long\def\parameterRobin{\sigma}

\global\long\def\quadraticForm{q_{\parameterRobin}}
 \global\long\def\sesquilinearForm{q_{\parameterRobin}'}

\global\long\def\Laplacian{\Delta}
 \global\long\def\RobinLaplacian#1#2{\Delta_{#1}^{#2}}
 \global\long\def\DirichletToNeumannOperator#1#2{\mathcal{D}_{#1}^{#2}}

\global\long\def\continuousFunctions#1{C^{0}(#1)}

\

\global\long\def\LTwoSpace#1{L^{2}(#1)}
 \global\long\def\HOneSpace#1#2{H_{#2}^{1}(#1)}

\global\long\def\volumeFormManifold#1{dvol_{#1}}
 \global\long\def\volumeFormBoundary#1{dvol_{#1}}

\global\long\def\innerProduct#1#2{\langle#1,#2\rangle}
 \global\long\def\norm#1#2{\left\Vert#1\right\Vert_{#2}}
 \global\long\def\absoluteValue#1{|#1|}

\global\long\def\boundary#1{\partial#1}
 \global\long\def\normal{\nu}
 \global\long\def\normalDerivative#1{\partial_{\nu}#1}

\global\long\def\eigenspace{E_{\parameterSteklov,\parameterRobin}}

\global\long\def\inducedRep#1#2#3{\mathrm{Ind}_{#1}^{#2}(#3)}
\global\long\def\restrictedRep#1#2#3{\mathrm{Res}_{#1}^{#2}(#3)}

\end{minipage}

\tableofcontents

\section{Introduction\label{sec:Introduction}}

\noindent Inverse spectral problems on compact Riemannian manifolds ask to what extent geometric and topological data are encoded in the spectra of natural operators.  There is an extremely rich literature of both positive and negative results in the case of the Laplace-Beltrami operator on compact manifolds, with Dirichlet or Neumann boundary conditions (or mixed conditions) imposed when the boundary is nonempty.  The literature for other natural operators lags behind.  The goal of this article is to show that most of the negative results for the Laplace-Beltrami operator in the literature, i.e., the constructions of manifolds whose Laplace-Beltrami operators are isospectral, are equally valid for other natural operators.  We were motivated primarily by the surge of interest in Steklov eigenvalue problems and the related ``sloshing problem'' on compact Riemannian manifolds with boundary, so we will focus primarily on these problems.  However, we will also comment on other eigenvalue problems.  

\subsection{Steklov eigenvalue problems.}\label{subsec.stek}  Let $(M,g)$ be a compact smooth Riemannian manifold with boundary, and let $\Delta$ be the associated Laplace-Beltrami operator.
For~$\alpha\in\mathbb{R}$ which is not in the spectrum of the Dirichlet Laplacian and for $\rho\in C^\infty(\partial M)$, the \emph{Steklov spectrum} of $M$ at frequency $\alpha$ with boundary density $\rho$, denoted by $\Stek_{\alpha}(M,g,\rho)$ or simply by $\Stek_\alpha(M,g)$ if $\rho\equiv 1$, is the collection of real numbers $\sigma$ for which there exists a nontrivial solution $u\in C^\infty(M)$ to the eigenvalue problem 
\begin{equation}
\begin{cases}\Laplacian u=\parameterSteklov\,u\text{ on }\manifold\smallsetminus\boundary{\manifold}\\
\quad\normalDerivative u=\parameterRobin\,\rho u\text{ on }\partial M,\label{eq:LB_eigenvalue_problem}
\end{cases}
\end{equation}
where $\normalDerivative u$ is the normal derivative of $u$ on the boundary.
(The problem is well-defined, since it was required that $\alpha$ not be a Dirichlet eigenvalue of $\Delta$.)  In two dimensions, $\Stek_0(M,g,\rho)$ corresponds to the collection of squares of eigenfrequencies of a drum
all of whose mass is distributed along the boundary according to the density $\rho$ (see \cite{LambertiProvenzano}).  When $\rho\equiv 1$, the Steklov spectrum $\Stek_\alpha(M,g)$  is precisely the eigenvalue spectrum of the \emph{Dirichlet-to-Neumann operator} $\D_\alpha^{M,g}\colon C^\infty(\partial M){}\to C^\infty(\partial M){}$.  This operator associates to a function $v\in C^\infty(\partial M)$ the normal derivative of the unique extension $V:M\to\R$ of $v$ to $M$ that satisfies $\Delta V=\alpha V$.  In particular, when $\alpha=0$, the extension $V$ is harmonic, so is just the solution of the Dirichlet problem with initial data $v$.  
We remark that if the boundary density function $\rho$ is merely $L^{\infty}$, then \eqref{eq:LB_eigenvalue_problem} is still a well-defined eigenvalue problem, although the eigenfunctions are merely $H^1$ rather than smooth, and the boundary condition in \eqref{eq:LB_eigenvalue_problem} is interpreted in the sense of the Sobolev trace.

The Steklov spectrum was first introduced by A. Steklov in 1902 and has since found many remarkable applications; see the historical article \cite{Notices}.   For example, by examining the dependence of $\Stek_\alpha(M,g)$ on the parameter $\alpha$, Friedlander~\cite{Friedlander1991} derived an inequality between the Neumann and
Dirichlet eigenvalues of bounded $C^{1}$-domains in~$\mathbb{R}^{n}$;
this inequality was extended to Lipschitz domains by Arendt and Mazzeo~\cite{ArendtMazzeo2012}.  The study of the Steklov spectrum has recently gained impetus; see, for example, \cite{Brock2001, CianciGir, ColGitGir, ColboisElGirouard2011, FraserSchoen2016, GirouardPolterovich2012, GirouardParnovskiPolterovichSher2014,  Jammes2014, PolterovichSher2015, Karp2017, YY}, and the excellent survey \cite{GP}.   E.g., $\Stek_0(M,g)$ is known to determine the dimension and volume of~$\boundary{\manifold}$, the geometry of $\boundary{\manifold}$ if $\dim(\manifold)=2$~\cite{GirouardParnovskiPolterovichSher2014}, whether a domain in $\R^2$ is a disk ~\cite{GirouardParnovskiPolterovichSher2014}, and whether a domain in $\R^3$ with connected boundary is a ball~\cite{PolterovichSher2015}.  

The so-called \emph{sloshing problem}, describing oscillations of a fluid in an open container, is the special case of the Steklov problem~(\ref{eq:LB_eigenvalue_problem}) in which $\rho$ takes on only the values 0 and 1:   $\rho\equiv 0$ on the walls of the container and $\rho\equiv 1$ on the free surface of the fluid.  

In dimension two, the Steklov spectrum $\Stek_0(M,g,\rho)$ is invariant under conformal changes of metric away from the boundary; i.e., if $g'=e^fg$ with $f\equiv 0$ on $\partial M$, then $\Stek_0(M,g,\rho)=\Stek_0(M,g',\rho)$.   In fact, we even have $\D_0^{(M,g')}=\D_0^{(M,g')}$.  (This is immediate from the fact that the Laplacian of $g'$ is related to that of $g$ by $\Delta'=e^{-f}\Delta$ in dimension two.   In higher dimensions, this equality fails.)  We will say that 
$(M,g,\rho)$ and $(M',g',\rho')$ are \emph{trivially Steklov isospectral} for $\alpha=0$ if there exists a diffeomorphism $F$ from $M$ to $M'$ intertwining $\rho$ and $\rho'$ such that either (i) $F:(M,g)\to (M',g')$ is an isometry or (ii) $\dim(M)=2$ and $F^*g'=e^f g$ with $f\vert_{\partial M}=0$.  We caution that such conformal changes of metric will in general affect $\Stek_\alpha(M,g,\rho)$ for $\alpha\neq 0$, even in dimension two.

In this article we adapt to the Steklov setting the two primary techniques for constructing Laplace isospectral manifolds: Sunada's technique \cite{Sunada1985} and the torus action method (see, e.g., \cite{Gordon1994,Gordon2001,Schueth2001a,Schueth2001}).  Both techniques yield pairs of Riemannian manifolds $M_1$ and $M_2$ with boundary that are simultaneously Dirichet and Neumann isospectral and that also satisfy $\Stek_\alpha(M_1,g_1)=\Stek_\alpha(M_2,g_2)$ for all $\alpha$ not in the Dirichlet spectrum.  Moreover, $\Stek_\alpha(M_1,g_1,\rho_1)=\Stek_\alpha(M_2,g_2,\rho_2)$ for a large family of pairs of densities $(\rho_1,\rho_2)$.  The Laplace-Beltrami operators on the boundaries are also isospectral. (In some, but not all cases, the boundaries are isometric.)  

 We illustrate these techniques with nontrivial examples:   
 
\begin{itemize}
\item Pairs of (nonplanar) flat Steklov isospectral surfaces embedded in $\R^3$ constructed via the Sunada method;
\item Continuous families of mutually Steklov isospectral nonflat metrics on a ball in $\R^n$ constructed by the torus action method.
 \end{itemize}
 
 Specializing to the sloshing problem, we obtain, for example,
 \begin{itemize}
 \item Pairs of planar domains that are isospectral for the sloshing problem.
 \end{itemize}
 
Referencing our results, the article \cite{ADGHRS} gives examples of Steklov isospectral orbifolds using the Sunada and torus action techniques.   Example 6.1 in the same article uses direct computation to give examples of orbifold quotients $\Gamma_1\bs B$ and $\Gamma_2\bs B$ of Euclidean balls with $\Stek_0(\Gamma_1\bs B)=\Stek_0(\Gamma_2\bs B)$.  Lemma 6.1 of ~\cite{ColboisElGirouard2011} establishes that cylinders over Laplace-Beltrami isospectral closed manifolds have the same Steklov spectrum, again with $\alpha=0$.  To our knowledge, these examples exhaust the nontrivial examples of Steklov  isospectral  manifolds in the literature.

There are various notions of Dirichlet-to-Neumann operator acting on the space of $p$-forms on the boundary of a manifold.  The definitions in \cite{RS} and \cite{Karp} (the latter being a modification of a definition in \cite{BelSh}) give operators with discrete spectrum.    The Sunada method goes through for these Steklov spectra on $p$-forms.  However, the torus action method does not.   (This is not unexpected: the torus action method for the Laplace-Beltrami operator produces manifolds that are isospectral on functions, but it does not establish isospectrality for the Hodge Laplacian on $p$-forms.)

\subsection{Robin eigenvalue problems}\label{subsec.rob}
The Robin boundary value problem is dual to the Steklov eigenvalue problem in the following sense:  Set $\rho\equiv 1$.   Fixing a given $\sigma\in\R$ and interpreting (\ref{eq:LB_eigenvalue_problem}) as an eigenvalue problem for an unknown $\alpha$ converts (\ref{eq:LB_eigenvalue_problem})
into an eigenvalue problem with Robin boundary conditions.   Since the Steklov isospectral manifolds that we construct satisfy $\Stek_\alpha(M_1,g_1)=\Stek_\alpha(M_2,g_2)$ for \emph{every} allowable choice of the parameter $\alpha$, they will also be isospectral for the Robin boundary value problem for every choice of the Robin parameter $\parameterRobin$.  See \cite{ArendtMazzeo2012} for historical comments on this relationship between the Steklov and Robin problems.

The Sunada and torus action methods work equally well for the mixed Robin--Neumann--Dirichlet eigenvalue problem.   This problem asks for which $\parameterSteklov\in\mathbb{R}$
there exists $u\in C^\infty({\manifold}{})$, with normal derivative
$\normalDerivative u\in C^\infty({\boundary{\manifold}}{})$,
such that
\begin{equation}
\Laplacian u=\parameterSteklov\,u\text{ on }\manifold\smallsetminus\boundary{\manifold},\quad u=0\text{ on }\boundaryDirichlet,\quad\normalDerivative u=0\text{ on }\boundaryNeumann,\quad\text{and}\quad\normalDerivative u=\parameterRobin\,u\text{ on }S.\label{eq:LB_eigenvalue_problem_with_RDN_boundary_conditions}
\end{equation} where $\partial M=S\union N\union D$ (set-theoretic disjoint union) and where $\sigma$ is again a fixed Robin parameter.   In case $D=\varnothing$, then the mixed Robin-Neumann problem is dual in the sense above to the Steklov problem with boundary density $\rho\equiv 1$ on $S$ and $\rho\equiv 0$ on $N$.   

\subsection{Other eigenvalue problems.} Both the Sunada method and the torus action method are very robust.   We remark without proof that both methods easily extend, for example, to poly-Laplacians
$\Laplacian^{m}$ with Dirichlet boundary conditions $u=\normalDerivative u=\normalDerivative{^{2}u}=\ldots=\normalDerivative{^{m-1}u}=0$
on $\boundary{\manifold}$, as in the clamped plate problem where
$m=2$.

\bigskip
The paper is organized as follows:  In Sections~\ref{sec:The_Sunada_method}
and~\ref{sec:The_torus_action_method}, we adapt the Sunada method and the torus action method, respectively, to the Steklov settings.  Examples constructed via the two methods are given in Section~\ref{sec:Examples}.  Finally, in Section~\ref{sec:density}, we construct Steklov isospectral boundary density functions: more precisely, we adapt both the Sunada method and the torus action method using an idea introduced by R. Brooks in order to construct pairs of boundary density functions $\rho_1$ and $\rho_2$ on a compact Riemannian manifold $M$ with boundary such that $\Stek_\alpha(M,\rho_1)=\Stek_\alpha(M,\rho_2)$ for all $\alpha$ not in the Dirichlet spectrum of $M$.

\section*{Acknowledgements} We thank Dorothee Schueth for suggesting Proposition~\ref{prop:Boundary_volume_form_preservation} and its proof, and we thank Leonid Friedlander and Rafe Mazzeo for informative conversations.

\section{The Sunada method\label{sec:The_Sunada_method}}

\noindent We adapt the Sunada method~\cite{Sunada1985} to the context
of the Steklov spectra.
\begin{defn}
\label{def:Gassmann_equivalence} Let $\group$ be a finite group.
Two subgroups $\subgroup$ and $\mate{\subgroup}$ of $\group$ are
called \emph{almost conjugate} or \emph{Gassmann equivalent}, if every $g\in\group$ has equally many conjugates in $\subgroup$ and~$\mate{\subgroup}$.\end{defn}

\begin{remark}\label{rem.rep} Gassmann used such almost conjugate subgroups of a finite group to exhibit examples of pairs of nonisomorphic algebraic number fields with the same arithmetic (i.e., the same Dedekind zeta function).  The formula for the character of an induced representation shows easily that $H$ and $H'$ are almost conjugate if and only if the representations of $G$ induced from the trivial one-dimensional representations of $H$ and $H'$ are equivalent: i.e., $\inducedRep{\subgroup}{\group}{\boldsymbol{1}_{\subgroup}}\cong \inducedRep{\mate{\subgroup}}{\group}{\boldsymbol{1}_{\mate{\subgroup}}}$, where $\boldsymbol{1}_{\subgroup}$ and $\boldsymbol{1}_{\mate{\subgroup}}$ denote the trivial one-dimensional representations of $\subgroup$ and $\mate{\subgroup}$, respectively.

\end{remark}
\begin{thm}[Sunada's Theorem adapted to the Steklov setting]
\label{thm:Sunada}Let $\subgroup$ and $\mate{\subgroup}$ be almost
conjugate subgroups of a finite group $\group$.  Assume that $G$ acts  by isometries
on a compact Riemannian manifold $\manifold$ with 
boundary and that the restriction of the action to the subgroups $H$ and $H'$ is free.  Let $\rho$ be an $L^\infty$, nonnegative, $G$-invariant function on $\partial M$.  Continue to denote by $g$ and $\rho$ the Riemannian metric and the function induced on each of the orbit spaces $\subgroup\backslash\manifold$ and $\mate{\subgroup}\backslash\manifold$ by $g$ and $\rho$.  Then 
$$\Stek_\alpha(\subgroup\backslash\manifold,g,\rho) =\Stek_\alpha(\mate{\subgroup}\backslash\manifold,g,\rho)$$
for all $\alpha$ not in the Dirichlet spectrum of $\subgroup\bsl\manifold$ and $\mate{\subgroup}\bs\manifold$.   (Sunada's original theorem guarantees that the two quotient manifolds are both Dirichlet and Neumann isospectral, so the allowable choices of $\alpha$ are the same in both cases.)  

\end{thm}

\begin{proof} Fix $\alpha$ and $\rho$ as in the theorem.  We will abuse language and refer to solutions $u$ of Equation~(\ref{eq:LB_eigenvalue_problem}) in the Introduction as \emph{$\sigma$-eigenfunctions} for the $(\alpha,\rho)$-Steklov problem on $M$. 

In what follows, if $\Gamma$ is any group acting linearly on a vector space $V$, we denote by $V^{\Gamma}$ the subspace of $\Gamma$-fixed vectors in $V$.

There are numerous simple and elegant proofs of Sunada's original theorem, some of which compare the dimension of each eigenspace in the two manifolds.  These proofs go through without change in our setting.  The $\sigma$-eigenfunctions for the $(\alpha,\rho)$-Steklov problem on each of the quotient manifolds $H\bsl M$ and $H'\bsl M$ pull back to $G$-invariant $\sigma$-eigenfunctions for the $(\alpha,\rho)$-Steklov problem on $M$.  Thus letting $E_\sigma\subseteq C^{\infty}(M)$ be the $\sigma$-eigenspace for the $(\alpha,\rho)$-Steklov problem on $M$, we need only show that the subspaces $E_\sigma^H$ and $E_\sigma^{H'}$ of $H$-inveriant and $H'$-invariant functions, respectively, have the same dimension.  Hence the proof of Theorem~\ref{thm:Sunada} reduces to the following lemma.    

\begin{lem}\label{lem.dim} Let $H$ and $H'$ be almost conjugate subgroups of a finite group $G$ and let $V$ be any vector space on which $G$ acts.  Then $\dim(V^H)=\dim(V^{H'})$.

\end{lem}

T. Sunada \cite{Sunada1985} gave an elementary proof of this lemma by a trace formula; see also \cite{BuserBook}, p. 295.  H. Pesce~\cite{Pesce1994} gave a representation theoretic proof by applying Remark~\ref{rem.rep} along with Frobenius reciprocity to obtain

\[
\dim(V^{\subgroup})=[\boldsymbol{1}_{\subgroup}:\restrictedRep{\subgroup}{\group}V]=[\inducedRep{\subgroup}{\group}{\boldsymbol{1}_{\subgroup}}:V],
\]
where $[U:W]$ denotes the multiplicity of the representation $U$
in $W$. Since $\inducedRep{\subgroup}{\group}{\boldsymbol{1}_{\subgroup}}$
and $\inducedRep{\mate{\subgroup}}{\group}{\boldsymbol{1}_{\mate{\subgroup}}}$
are equivalent, it follows that $\dim(V^{\subgroup})=\dim(V^{\mate{\subgroup}})$.
\end{proof}

\begin{rems}\label{rem.Sunada} We note a couple of features of the Sunada construction.

\begin{enumerate}
\item\label{trans} Lemma~\ref{lem.dim} says that the vector spaces $V^H$ and $V^{H'}$ are isomorphic.   In fact, the equivalence $\tau$ between the induced representations $\inducedRep{\subgroup}{\group}{\boldsymbol{1}_{\subgroup}}$ and $\inducedRep{\mate{\subgroup}}{\group}{\boldsymbol{1}_{\mate{\subgroup}}}$ actually yields an explicit and natural isomorphism $\tau^{\sharp}:V^{H'}\to V^H$, which Peter Buser and Pierre B\'erard \cite{Buser1986,Berard1992} called \emph{transplantation}.  See also \cite{Zelditch}, \cite{BrooksGornetPerry}, \cite{GMW}. 
 
\item If $H$ and $H'$ are conjugate subgroups of $G$, then the resulting quotient manifolds $H\bsl M$ and $H'\bsl M$ are isometric.  Even when $H$ and $H'$ are not conjugate, the quotient manifolds may be accidentally isometric.  Thus one must always verify nontriviality when using Sunada's technique (in fact, when using any of the  known techniques for constructing isospectral manifolds).
\end{enumerate}
\end{rems}

More important for our purposes is:
 
\begin{remark}\label{rem.OrbSunada}
One may drop the hypothesis that $H$ and $H'$ act freely.  The resulting quotients $\subgroup\bsl\manifold$ and $\mate{\subgroup}\bsl\manifold$ will then be Steklov isospectral good Riemannian orbifolds.  (A \emph{good orbifold} is the orbit space $\orb=\Gamma\bsl M$ of a manifold by a smooth discrete group action satisfying the condition that the isotropy group at any point is finite.  A function on $\mathcal{O}$ is said to be \emph{smooth} if its pullback to $M$ is smooth.  If $g$ is a Riemannian metric on $M$ and $\Gamma$ acts by isometries, then $g$ gives $\mathcal{O}$ the structure of a Riemannian orbifold.  The associated Laplacian $\Delta_\orb: C^\infty(\orb)\to C^\infty(\orb)$ is defined by $\pi^*\circ \Delta_\orb=\Delta_M\circ\pi^*$ where $\pi:M\to\orb$ is the projection.)  We will apply the orbifold version in Example~\ref{exa.domains} when we construct planar domains that are isospectral for the sloshing problem.
\end{remark}
 
\begin{oep}\label{rem.sun} (i) There are various notions in the literature of  a Dirichlet-to-Neumann operator acting on the space of smooth differential $p$-forms on $\partial M$ where $M$ is a smooth compact Riemannian manifold with smooth boundary.  The notions of Dirichlet-to-Neumann operator on forms defined by S. Raulot and A. Savo \cite{RS} and by Karpukhin \cite{Karp} have discrete spectra.  Using either of these definitions of Steklov spectrum on $p$-forms, the hypotheses of Theorem~\ref{thm:Sunada} (with $\rho\equiv 1$) guarantee that the manifolds $H\bsl M$ and $H'\bsl M$ have the same Steklov spectra on $p$-forms, for all $p$.
 
 (ii) As noted in the introduction, taking $\rho\equiv1$ in Theorem~\ref{thm:Sunada} immediately yields isospectrality of the Robin problems on $H\bsl M$ and $H'\bsl M$ for every choice of Robin parameter.   Alternatively, one can prove the Robin isospectrality directly using the same method as in the proof of Theorem~\ref{thm:Sunada}.

Moreover, one can easily modify Theorem~\ref{thm:Sunada} to address mixed Robin-Neumann-Dirichlet problems.    One assumes that $\partial M=\partial_R M\union \partial_N M\union \partial _D M$, where each of the three subsets is $G$-invariant and where the decomposition is sufficiently nice so that the mixed Robin-Neumann-Dirichlet problem, in which Robin, Neumann, and Dirichlet conditions are imposed on $\partial_R M,\, \partial_N M$ and $\partial _D M$, respectively, is well-defined with discrete spectrum.   Then the mixed problems on $H\bsl M$ and $H'\bsl M$ are isospectral, where the respective boundary conditions are imposed on $H\bsl(\partial_R M),\, H\bsl(\partial_N M)$, and $H\bsl(\partial _D M)$ and similarly for $H'$.

\end{oep}

\section{The torus action method\label{sec:The_torus_action_method}}

The torus action method was developed to construct Riemannian manifolds that have the same Laplace spectrum but that are not even locally isometric.   There are several versions, e.g., ~\cite{Gordon1994,Gordon2001,Schueth2001a, Schueth2001}.   We first state the version in \cite{Schueth2001} and then adapt it to the Steklov setting.  

In the following, a \emph{torus} always means a nontrivial, compact, connected, abelian Lie group. Let~$\torus$ be a torus acting effectively by isometries on a compact, connected Riemannian manifold~$\manifold$.  The union of those orbits on which $\torus$ acts freely is an open, dense submanifold of $\manifold$ (see \cite{Bredon}) that we will denote by $\principleOrbits{\manifold}$ ; it carries the structure of a principal $\torus$-bundle.  

\begin{thm}\label{thm.schueth}\cite{Schueth2001} Let $T$ be a torus which acts effectively on two compact,
connected Riemannian manifolds $(M,g)$ and $(M',g')$ by isometries.
For each subtorus $W\subset T$ of codimension one, suppose that there exists
a $T$-equivariant diffeomorphism $F_W: M\to M'$ such that
\begin{enumerate}
\item $F_W:M\to M'$ is volume-preserving; i.e., $F_W^* dvol_{M'}= dvol_M$ where $dvol_M$ and $dvol_{M'}$ are the Riemannian volume densities of $M$ and $M'$;
\item  $F$ induces an isometry $\oF_W: (W\bsl\hm, g_W)\to(W\bsl\widehat{M'},g'_W)$ where $g_W$ and $g'_W$ are the metrics induced by $g$ and $g'$ on the quotients.
\end{enumerate}
Then $(M,g)$ and $(M',g')$ are
isospectral.  Moreover,  if the manifolds have boundary, then they are both Dirichlet and Neumann isospectral.

\end{thm}

We now adapt this method to the Steklov setting.

\begin{thm}
\label{thm:Torus_method}
Let $T$ be a torus which acts isometrically and effectively on two compact,
connected Riemannian manifolds $(M,g)$ and $(M',g')$ with boundary.  Let $\rho\in L^\infty(\partial M)$ and $\rho'\in L^\infty(\partial M')$ be $T$-invariant.
For each subtorus $W\subset T$ of codimension one, suppose that there exists
a $T$-equivariant diffeomorphism $F_W: M\to M'$ such that
\begin{enumerate}
\item\label{F_WVolPres} 
$F_W:M\to M'$ is volume-preserving;
\item\label{F_WBoundaryVolPres} 
$F\vert_{\partial M}:\partial M\to\partial M'$ is volume-preserving, i.e., $F_W^* dvol_{\partial M'}= dvol_{\partial M}$;
\item\label{F_WPresDensity} $F_W^*\rho'=\rho$;
\item\label{F_WIsomOnQuots} $F_W$ induces an isometry $\oF_W: (W\bsl\hm, g_W)\to(W\bsl\widehat{M'},g'_W)$, where $g_W$ and $g'_W$ are the metrics induced by $g$ and $g'$ on the quotients.
\end{enumerate}

Then for each $\alpha$ not in the Dirichlet spectrum of $(M,g)$, we have 
\begin{equation}\label{eq.same}\Stek_\alpha(M,g,\rho)=\Stek_\alpha(M',g',\rho').\end{equation}
(Theorem~\ref{thm.schueth} guarantees that the two quotient manifolds are Dirichlet isospectral, so the allowable choices of $\alpha$ are the same in both cases.)
\end{thm}

Before proving Theorem~\ref{thm:Torus_method}, we recall the variational characterization of the eigenvalues in $\Stek_\alpha(M,g,\rho)$.   First recall that the boundary restriction map that takes $u\in H^1(M)\cap C^0(M)$ to $u|_{\partial M}$ extends to the compact trace operator $\Tr: H^1(M)\to L^2(\partial M)$.  We write $u|_{\partial M}=\Tr(u)$.  Define
\begin{equation}\label{eq.ray} R_{M,\alpha,\rho}(u)=\frac{\int_M\,\Vert\nabla u\Vert^2\,dvol_M-\alpha\int_M\,  u^2\,dvol_M}{\int_{\partial M}\,u|_{\partial M}^2\,\rho\,dvol_{\partial M}}\end{equation}
Denoting the eigenvalues in $\Stek_\alpha(M,g,\rho)$ as 
$$0=\sigma_0<\sigma_1\leq \sigma_2\leq\dots,$$
we have 
\begin{equation}\label{minmax}\sigma_k=\inf_{E_k(M,\rho)}\sup_{0\neq u\in E_k}\,R_{M,\alpha,\rho}(u)\end{equation}
where the infimum is over all $k$-dimensional subspaces $E_k(M,\rho)$ of $H^1(M)$ consisting of functions whose restrictions to the boundary are $\rho$-orthogonal to the constant functions, i.e $\int_{\partial M}\,u|_{\partial M}\,\rho\,dvol_{\partial M}=0$.  
\begin{proof}

[Proof of Theorem \ref{thm:Torus_method}] We adapt the proof of~\cite[Theorem 1.4]{Schueth2001}.  For $\subtorus<\torus$ any subtorus, let $H^1(M)^W \subset H^1(M)$, $L^2(M)^W \subset L^2(M)$, $H^1(M')^W \subset H^1(M')$, and $L^2(M')^W \subset L^2(M')$
denote the subspaces of $\subtorus$-invariant functions.  By Fourier decomposition with respect to the isometric action of $T$, we have
\begin{equation}\label{eq.decomp}H^1(M)=H^1(M)^T\,\oplus\,\bigoplus_W\,(H^1(M)^W\ominus H^1(M)^T)\end{equation}
and 
\begin{equation}\label{eq.decomp2}L^2(\partial M)=L^2(\partial M)^T\,\oplus\,\bigoplus_W\,(L^2(\partial M)^W\ominus L^2(\partial M)^T)\end{equation}
where the sum is over all subtori $W$ of $T$ of codimension one. Multiplication by the $T$-invariant density $\rho$ preserves each of the subspaces $L^2(M)^T$ and $L^2(M)^W$.  Moreover the trace operator $\mathrm{Tr}: H^1(M)\to L^2(\partial M)$ respects these decompositions.  Analogous statements  hold with $M$ replaced by $M'$.

As shown in~\cite{Schueth2001}, conditions \eqref{F_WVolPres} and \eqref{F_WIsomOnQuots} of Theorem~\ref{thm:Torus_method} imply that if $\subtorus$ is a subtorus of $T$ of codimension at most one and $u\in H^1(M')^W$, then
\begin{equation}\label{eq.h1}
\norm{\diffeo^{*}u}{\HOneSpace{\manifold}{}}=\norm u{\HOneSpace{\mate{\manifold}}{}}\qquad\text{and also}\qquad\norm{\diffeo^{*}u}{\LTwoSpace{\manifold}}=\norm u{\LTwoSpace{\mate{\manifold}}}.
\end{equation}
The first of the equations in~\eqref{eq.h1}, the $T$-equivariance of the maps $F_W$, and Equation~(\ref{eq.decomp}) yield an isomorphism 
$$\tau: H^1(M')\to H^1(M)$$ 
given by 
$$\tau=F_T^*\,\oplus\,\bigoplus_W\,F_W^*.$$
Hypothesis \eqref{F_WBoundaryVolPres} of the theorem and Equation~(\ref{eq.decomp2}) similarly yield an isomorphism
$$\tau_\partial:=F_T^*\,\oplus\,\bigoplus_W\,F_W^*: L^2(\partial M')\to L^2(\partial M)$$
and the diagram
$$\xymatrix{
{H^1(M')}\ar^{\tau}[r]\ar_{\Tr}[d]&{H^1(M')}\ar^{\Tr}[d]\\
{L^2(\partial M')}\ar_{\tau_{\partial}}[r]&{L^2(\partial M)}
}$$
commutes.

Hypotheses \eqref{F_WBoundaryVolPres} and \eqref{F_WPresDensity} of the theorem guarantee for each $k=1,2,\dots$ that $\tau$ maps $E_k(M',\rho')$ to $E_k(M,\rho)$ and that the denominators in the Rayleigh quotients $R_{M,\alpha,\rho}(\tau(u))$ and $R_{M',\alpha,\rho'}(u)$ coincide for each $u\in E_k(M',\rho')$.
The pair of equalities~\eqref{eq.h1} imply that the
numerators in $R_{M,\alpha,\rho}(\tau(u))$ and $R_{M',\alpha,\rho'}(u)$ also agree, and the theorem follows from Equation~(\ref{minmax}).\end{proof}

Although condition \eqref{F_WBoundaryVolPres} in Theorem~\ref{thm:Torus_method} does not appear in Theorem~\ref{thm.schueth} or in any of the other versions of the torus action method, it is actually satisfied in all of the examples that have been constructed thus far by these methods, as will be explained
in Section~\ref{sec:Examples}.  Moreover, the version of the torus action method in \cite[Theorem 1.2]{Gordon2001} includes a hypothesis that the principal $T$-orbits be dense in $\boundary{\manifold}$ and $\boundary{\mate{\manifold}}$ in order to produce Neumann isospectral manifolds; this condition is stronger than condition \eqref{F_WBoundaryVolPres} in the following sense.
\begin{prop}
\label{prop:Boundary_volume_form_preservation}  Let $\manifold$ and $\mate{\manifold}$ be compact, connected, orientable Riemannian manifolds with a faithful isometric action by a torus $\torus$ satisfying conditions~\eqref{F_WVolPres}  and~\eqref{F_WIsomOnQuots} of Theorem~\ref{thm:Torus_method}.  If~$\principleOrbits M\cap\boundary{\manifold}$ is dense in $\boundary{\manifold}$, then condition \eqref{F_WBoundaryVolPres} of Theorem~\ref{thm:Torus_method} is satisfied as well.\end{prop}

\begin{proof} Since the manifolds are orientable, condition~\eqref{F_WVolPres} says that $F_W$ pulls back the Riemannian volume form $dvol_{M'}$ of $M'$ to that of $M$.

By continuity, it suffices to show that condition \eqref{F_WBoundaryVolPres} holds at each $p\in\principleOrbits M\cap\boundary{\manifold}$.  Let $p\in \hm\cap \partial M$ and let  $p'=F_W(p)$.   The $T$-equivariance of $F_W$ guarantees that $p'\in\hm'$.   Let $\nu$ and $\nu'$ denote the outward unit normals to $\partial M$ and $\partial M'$ at $p$ and $p'$, respectively, and let $i:\partial M\to M$ and $i':\partial M'\to M'$ be the inclusion maps.  The facts that $F_W$ is an isometry and that the action of $W$ on $M$ and $M'$ preserves the boundaries imply that $\nu'-{F_W}_*(\nu)$ is tangent to $\partial M'$ and hence  
\begin{equation}\label{eq.tan} (i')^*({F_W}_*(\nu)\,\lrcorner\,  dvol_{M'})=(i')^*(\nu'\,\lrcorner\, dvol_{M'}) = dvol_{\partial M'}
 \end{equation}
 By condition (1) of Theorem~\ref{thm:Torus_method}, we have 
 \begin{equation}\label{eq.pull}{F_W}_*(\nu)\,\lrcorner\,  dvol_{M'}={F_W}_*(\nu)\,\lrcorner\, (F_W^{-1})^* dvol_{M}=(F_W^{-1})^*(\nu\,\lrcorner\,dvol_{M})\end{equation}
 Since $F_W\circ i=i'\circ F_W$, Equations~\ref{eq.tan} and \ref{eq.pull} yield
$$F_W^*(dvol_{\partial M'})=F_W^*\circ(i')^*\circ (F_W^{-1})^*(\nu\,\lrcorner\,dvol_{M})=i^*(\nu\,\lrcorner\,dvol_{M})=dvol_{\partial M}.$$

\end{proof}

\section{Examples\label{sec:Examples}}

\subsection{Examples using the Sunada technique}

There is a wealth of examples of Dirichlet or Neumann isospectral manifolds that have been constructed by the Sunada method and its various generalizations; see \cite{G:SunadaTwoDecades} and references therein.   The original Sunada technique has yielded, for example, isospectral flat surfaces embedded in $\R^3$ \cite{Buser1988} and large finite families of mutually isospectral Riemann surfaces \cite{BGG}, which can be easily modified to produce families of mutually isospectral hyperbolic surfaces with boundary.   All examples of isospectral manifolds with boundary constructed by the original Sunada technique are also Steklov isospectral.  

There are various generalizations of Sunada's theorem, surveyed in \cite{G:SunadaTwoDecades}, not all of which go through directly for the Steklov spectrum.   For example, the pair of Neumann isospectral flat surfaces with boundary constructed in \cite{BerardWebb1995} (one orientable, the other nonorientable) using the orbifold version of Sunada's Theorem are not Steklov isospectral, since one of the manifolds has four boundary components while its isospectral companion has only three boundary components.  Yet the number of boundary components of a surface is determined by the Steklov spectrum (see \cite{GirouardParnovskiPolterovichSher2014}).  See \ref{exa.domains} for some further comments.

In this subsection we illustrate the Sunada method with just a sampling of the many examples.

\subsubsection{Steklov isospectral flat surfaces embedded in $\R^3$}\label{exa.buser}

In \cite{Buser1988}, Peter Buser introduced the use of Schreier graphs to construct isospectral manifolds via Sunada's technique and illustrated the method by constructing a pair of Dirichlet and Neumann isospectral flat surfaces with boundary in $\R^3$. 

For the reader's convenience, we briefly review Buser's construction before addressing the Steklov setting.   Recall that if $G$ is a finite group and $S=\{s_{1},s_{2},\ldots,s_{n}\}$ is a set of nonidentity elements generating $G$, the \emph{Cayley graph} $\Gamma(\group,S)$ is the $n$-regular edge-colored directed graph whose vertices are the elements of~$\group$, and whose $i$-colored edges encode right multiplication by the generators $s_{i}$.  More precisely, there is an $i$-colored edge from $g$ to~$g'$ if and only if $g s_i=g'$.  The group $\group$ acts transitively and faithfully on $\Gamma(\group,S)$ by left multiplication.  If $H$ is a subgroup of $G$, then the \emph{Schreier graph} $\Gamma(H\bs G, S)$ is the quotient of $\Gamma(\group,S)$ by the action of $H$.  Equivalently, the vertices of the Schreier graph correspond to the elements of the space of right-cosets $H\bs G$ and the edges indicate the right action of the elements of $S$ on $H\bs G$.  The graph theoretic version of Sunada's Theorem says that if $H_1$ and $H_2$ are almost conjugate subgroups of $G$, then for any fixed choice of generating set $S$, the adjacency operators (or Laplacians or other natural operators) associated with the Schreier graphs $\Gamma(H_1\bs G,S)$ and $\Gamma(H_2\bs G,S)$ are isospectral.

To construct a manifold from a Schreier graph, Buser chooses a basic tile $T$, whose piecewise-smooth boundary contains $2n$ disjoint line segments called sides, labelled $s_1$, $s_1^{-1}$, $s_2$, $s_2^{-1},\ldots,s_n,s_n^{-1}$.   
Sides $s_i$ and $s_i^{-1}$ are required to have the same length. The sides need not exhaust the entire boundary of the tile.  To construct a manifold $M(H\bs G,S)$, consider a collection of $[G:H]$ identical tiles, labelled by the elements of $H\bs G$, whose sides are glued together in pairs according to the pattern encoded by the Schreier graph.  More precisely, side $s_{i}$ of tile $Hg$ is glued to side $s_{i}^{-1}$ of tile $H g s_i$.  Similarly, one uses the Cayley graph $\Gamma(G,S)$ to construct a manifold $M(G,S)$.  Observe that $G$ acts on $M(G,S)$ on the left, and that $M(H\bs G,S)=H\bs M(G,S)$.   Let $\partial _0(T)$ denote the complement in $\partial T$ of the union of the sides $s_{1},s_{1}^{-1},s_{2},s_{2}^{-1},\ldots,s_{n},s_{n}^{-1}$.  Buser arbitrarily chooses boundary conditions on $\partial_0 T$.   The boundary conditions chosen on $\partial_0 T$ then determine the boundary conditions on the manifold $M(G,S)$ and on $M(H\bs G,S)$ for any subgroup $H<G$.  

Now suppose that $H_1$ and $H_2$ are almost conjugate subgroups of $G$.  Then, as observed by Buser, Sunada's Theorem immediately yields isospectrality of $M_1:=M(H_1\bs G,S)$ and $M_2:=M(H_2\bs G,S)$ with respect to the given boundary conditions.

Moving to our setting, we instead choose arbitrarily an $L^\infty$ density function $\rho$ on $\partial_0 T$, thus giving rise to a density function, still denoted $\rho$, on the boundaries of $M:=M(G,S)$ and $M_i$, $i=1,2$.  The density on $\partial M(G,S)$ is $G$-invariant and Theorem~\ref{thm:Sunada} yields
$$\Stek_\alpha(M_1,\rho)=\Stek_\alpha(M_2,\rho).$$

It is easy to construct an abundance of examples this way.  For a concrete example, we consider the pair of flat surfaces in $\R^3$ given by Buser in \cite{Buser1988}.  In this example,  $\group=\operatorname{GL}(3,\Z_2)$, $H_1$ is the subset of matrices with first row $(1,0,0)$, and $H_2=H_1^{t}$ is the subset consisting of transposes of elements of $H_1$.  The two subgroups $H_1$ and $H_2$ are almost conjugate in $G$ (each element of $H_1$ is similar to its transpose in $H_2$) and have index 7 in $G$.  Buser's surfaces are obtained by using a particular generating set $S=\{a,b\}$ of order 2 and the basic tile shown in Figure~\ref{fig:Basic_tile}.   (Ignore for now the dashed line in Figure~\ref{fig:Basic_tile}; it will be used in the next example.)  Buser actually used a cross-shaped tile; we have smoothed out the corners of the tile so that the resulting isospectral surfaces $M_1$ and $M_2$ are smooth.  

We have not included a picture of the two surfaces here.  However, Figure~\ref{fig:RS_isospectral_planar_domains} shows the quotient of each of the two surfaces by a reflection.  To visualize the original surfaces, simply double the two domains in the figure across the part of the boundary indicated by double lines.  Alternatively, see  \cite{Buser1988}, where the surfaces constructed with a cross-shaped tile are drawn.

The surfaces are easily seen to be nonisometric; in fact they have different diameter.   Since we are in dimension two, we also verify that they are not trivially Steklov isospectral when $\alpha=0$ by showing that $M_2$ is not isometric to the surface $M_1$ endowed with a metric $e^f g_E$, where $g_E$ is the Euclidean metric and where the conformal factor $f$ vanishes on the boundary.  Recall that the scalar curvature of $e^f g_E$ is $4e^{-f}\Delta f$, where $\Delta$ denotes the Euclidean Laplacian.  Noting that $M_2$ is flat, we conclude that $f$ must be a harmonic function.   Since $f$ vanishes on the boundary, $f$ must be identically zero.   Thus no such conformal equivalence exists and the surfaces are nontrivially Steklov isospectral.

\noindent 
\begin{figure}
\noindent \begin{centering}
\hfill{}\subfloat[Buser's tile\label{fig:Basic_tile}]{\noindent \protect\begin{centering}
\begin{minipage}[b]{0.2\columnwidth}%
\noindent \protect\begin{center}
\raisebox{0.2mm}
{\begin{tikzpicture}[scale=0.4,thick,rotate=45]

\draw (1.77,4.77) node {$a$};
\draw (5.25,1.6) node {$a^{-1}$};
\draw (1.9,-0.92) node {$b$};
\draw (-1,2) node {$b^{-1}$};

\draw[dashed,thin] ($(0,0) + (45:1)$) -- ($(4,4) + (225:1)$);
 \draw ( 1, 0) arc (0:90:1);
 \draw ( 3, 0) arc (180:90:1);
 \draw ( 1, 4) arc (0:-90:1);
 \draw ( 3, 4) arc (180:270:1);
  
 \draw[double] (0.97,0) -- (3.03,0);
 \draw[double] (0,0.97) -- (0,3.03);
 \draw[double] (0.97,4) -- (3.03,4);
 \draw[double] (4,0.97) -- (4,3.03);

\end{tikzpicture}}\protect
\par\end{center}%
\end{minipage}\protect
\par\end{centering}

}\hfill{}\subfloat[Steklov-Neumann and Robin-Neumann isospectral planar domains\label{fig:RS_isospectral_planar_domains}]{\noindent \protect\begin{centering}
\begin{minipage}[b]{0.7\columnwidth}%
\noindent \protect\begin{center}
\begin{tikzpicture}[scale=0.4,thick,rotate=45]

 \draw[dashed,thin] ($(3,3) + (45:1)$) -- ($(7,7) + (225:1)$);
 \draw[dashed,thin] ($(7,3) + (45:1)$) -- ($(11,7) + (225:1)$);
 \draw[dashed,thin] (8,3) -- (10,3);
 \draw[dashed,thin] (7,4) -- (7,6);
 \draw[dashed,thin] (3,4) -- (3,6);
 \draw[dashed,thin] (0,7) -- (2,7);

 \draw (7,0) arc (90:45:1);
 \draw (11,4) arc (90:225:1);
 \draw (10,7) arc (180:270:1);
 \draw (6,7) arc (180:360:1);
 \draw (3,8) arc (90:360:1);
 \draw (3,10) arc (270:225:1);
 \draw ($(-1,7)+(45:1)$) arc (45:-45:1);
 \draw ($(3,3)+(135:1)$) arc (135:0:1);
 \draw (7,2) arc (-90:180:1);

 \draw[double] ($(7,-1)+(45:0.97)$) -- ($(11,3)+(225:0.97)$);
 \draw[double] (11,3.97) -- (11,6.03);
 \draw[double] (10.03,7) -- (7.97,7);
 \draw[double] (6.03,7) -- (3.97,7);
 \draw[double] (3,7.97) -- (3,10.03);
 \draw[double] ($(3,11)+(225:0.97)$) -- ($(-1,7)+(45:0.97)$);
 \draw[double] ($(-1,7)+(-45:0.97)$) -- ($(3,3)+(135:0.97)$);
 \draw[double] (3.97,3) -- (6.03,3);
 \draw[double] (7,-0.03) -- (7,2.03);
\end{tikzpicture}\hspace{10mm}\begin{tikzpicture}[scale=0.4,thick,rotate=45]

 \draw[dashed,thin] ($(-1,7) + (45:1)$) -- ($(3,11) + (225:1)$);
 \draw[dashed,thin] ($(3,3) + (45:1)$) -- ($(7,7) + (225:1)$);
 \draw[dashed,thin] (4,3) -- (6,3);
 \draw[dashed,thin] (7,4) -- (7,6);
 \draw[dashed,thin] (3,4) -- (3,6);
 \draw[dashed,thin] (0,7) -- (2,7);

 \draw ($(3,3)+(135:1)$) arc (135:-45:1);
 \draw (7,0) arc (90:135:1);
 \draw (7,2) arc (270:0:1);
 \draw (10,3) arc (180:135:1);
 \draw (6,7) arc (180:315:1);
 \draw (3,8) arc (90:360:1);
 \draw (3,10) arc (270:180:1);
 \draw (0,11) arc (360:270:1);
 \draw (-1,8) arc (90:-45:1);

 \draw[double] ($(-1,7)+(-45:0.97)$) -- ($(3,3)+(135:0.97)$);
 \draw[double] ($(3,3)+(-45:0.97)$) -- ($(7,-1)+(135:0.97)$);
 \draw[double] (7,-0.03) -- (7,2.03);
 \draw[double] (7.97,3) -- (10.03,3);
 \draw[double] ($(7,7)+(315:0.97)$) -- ($(11,3)+(135:0.97)$);
 \draw[double] (6.03,7) -- (3.97,7);
 \draw[double] (3,7.97) -- (3,10.03);
 \draw[double] (2.03,11) -- (-0.03,11);
 \draw[double] (-1,10.03) -- (-1,7.97);
\end{tikzpicture}\protect
\par\end{center}%
\end{minipage}\protect
\par\end{centering}

}\hfill{}
\par\end{centering}

\caption{Neumann conditions are imposed on all straight boundary parts (double-lined).  The domains arise from~\cite[Figure 7]{GordonWebbWolpert1992} by
using tiles as in Figure~\ref{fig:Basic_tile}.}\label{Figure1}

\end{figure}
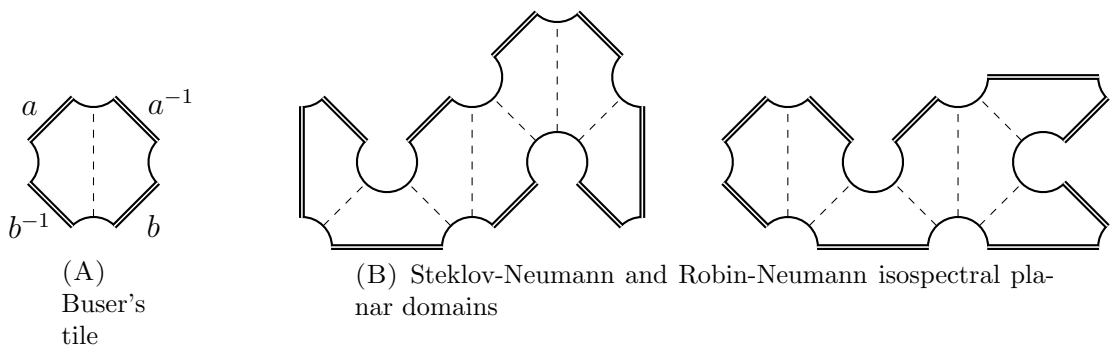

\subsubsection{Planar domains with isospectral sloshing problems} \label{exa.domains}

The first examples of isospectral planar domains \cite{GordonWebbWolpert1992} arose from the observation that the two isospectral flat surfaces $M_i$, $i=1,2$, described in the previous example each admit an isometric involution $\beta_i$, covering the symmetry $\beta_0$ of the basic tile in Figure~\ref{fig:Basic_tile} given by reflection across the dashed line.  The quotients of the surfaces by the involutions, shown in Figure~\ref{fig:RS_isospectral_planar_domains}, are both Dirichlet and Neumann isospectral.  As we will explain below, the version of Sunada's technique used to prove isospectrality does not yield Steklov isospectrality of these domains except in the special case that the density $\rho$ is identically zero on the part of the boundary indicated by double lines (the straight segments of the boundary) in Figure~\ref{fig:RS_isospectral_planar_domains}.   However, if we choose $\rho$ to be zero on this part of the boundary and $\rho\equiv 1$ on the curved edges, then we do obtain isospectrality for the mixed Neumann-Steklov problem (the sloshing problem).   One can also make a more general choice of $\rho$ on the curved parts of the boundary as long as consistency is maintained among the various tiles.

The proof in \cite{GordonWebbWolpert1992} of Neumann isospectrality of the planar domains goes as follows:  The involutive isometries $\beta_i$, $i=1,2$, lift to an involutive isometry $\beta$ of the covering manifold $M=M(G,S)$.  The isometry $\beta$ normalizes the group $G$ and each of the subgroups $H_i$, $i=1,2$.  The groups $\tilde{H_1}:=H_1\rtimes\langle\beta\rangle$ and $\tilde{H_2}:= H_2\rtimes\langle\beta\rangle$ are almost conjugate subgroups of $\tilde{G}:=G\rtimes\langle\beta\rangle$.  The group $\tilde{G}$ does not act freely on $M$.   However, we may apply the orbifold version of Sunada's Theorem as in Remark~\ref{rem.OrbSunada} to conclude that the quotients $\tilde{H_1}\bs M$ and $\tilde{H_2}\bs M$ are isospectral  orbifolds.  The underlying spaces of these orbifolds are the domains in Figure~\ref{fig:RS_isospectral_planar_domains}.  The singular sets of these orbifolds consist of the doubled line segments in Figure~\ref{fig:RS_isospectral_planar_domains},  which are reflector edges where the isotropy group has order $2$.  (Note that these line segments lift to interior segments of $M$, not to boundary edges.)    By the definition of smooth functions and of the Laplacian on these orbifolds (see Remark~\ref{rem.OrbSunada}), the isospectrality of the two orbifolds is equivalent to isospectrality of the underlying planar domains with Neumann boundary conditions placed on the doubled line segments of the boundary and  whatever boundary conditions on the curved edges were chosen on the curved edges of the basic tile $T$ used to construct $M$.   

If we choose the boundary density function $\rho\equiv 1$ on the boundary of the basic tile, the same argument yields the Steklov isospectrality of the two orbifolds, which in turn corresponds to isospectrality for the sloshing problem on the two underlying planar domains.

\begin{remark}\label{rem.trans2} We have summarized the original proof of the isospectrality of the planar domains in order to make clear the reason we can only get sloshing isospectrality rather than more general Steklov isospectrality of the planar domains.   However, transplantation as in Remark~\ref{trans} yields a very simple proof by picture of the sloshing isospectrality.  


\end{remark}

\subsubsection{Mixed Robin-Neumann-Dirichlet and Steklov-Neumann-Dirichlet isospectral domains}\label{exa.RND}
Levitin, Parnovski and Polterovich \cite{LPP} constructed examples of pairs of domains that are isospectral with mixed boundary conditions, including a pair consisting of a triangle and a square, whose isospectrality cannot be explained directly by Sunada's technique but can be shown by an explicit transplantation of eigenfunctions.  Later Band and Parzanchevsky \cite{BP} gave a representation theoretic explanation, which was further developed and applied systematically in Herbrich \cite{Herbrich2011}.  

One can similarly use transplantation directly to obtain domains that are isospectral for the mixed Robin-Neumann-Dirichlet and mixed Steklov-Neumann-Dirichlet problems.  We give two examples here, both obtained by modifying the construction of the isospectral triangle and square in \cite{LPP}.  The triangle and square in \cite{LPP} are each constructed by gluing together two copies of an isosceles right triangle (the basic tile); they are glued along the hypotenuse to obtain the square and along one of the legs to obtain the triangle in the isospectral pair.   Figure~\ref{fig:RS_isospectral_square_and_triangle} shows two modifications of their construction, both obtained by cutting out a half disk from the basic tile.   

For the mixed Robin-Neumann-Dirichlet problem, we impose Robin boundary conditions --- with the same Robin parameter on both domains in each pair --- on the curved part of the boundary indicated by a solid line in the figures, Neumann conditions on the part of the boundary indicated by doubled lines, and Dirichlet conditions on the part indicated by dashed lines.  With these boundary conditions we claim that $M$ is isospectral to $M'$ and $P$ is isospectral to $P'$.

Let $u$ be an eigenfunction for the mixed problem on $M$, say with eigenvalue $\lambda$, and denote by $u_1$ and $u_2$ the restrictions of $u$ to the two tiles making up $M$ as in Figure~\ref{fig:RS_isospectral_square_and_triangle}.
We transplant $u$ to an eigenfunction $u'=T(u)$ on $M'$ whose restrictions $u_1'$ and $u_2'$ to the two tiles of $M'$ as in Figure~\ref{fig:RS_isospectral_square_and_triangle} are given by

\begin{equation}
\left(\begin{array}{c}
\mate u_{1}\\
\mate u_{2}
\end{array}\right)=\frac{1}{\sqrt{2}}\left(\begin{array}{cc}
1 & -1\\
1 & 1
\end{array}\right)\left(\begin{array}{c}
u_{1}\\
u_{2}
\end{array}\right).\label{eq:Transplantation_in_matrix_form}
\end{equation}
In writing $u_{1}\pm u_{2}$, we implicitly identify the tiles underlying
$u_{1}$ and $u_{2}$, which involves a reflection in the dotted diagonal
of $\manifold$.   To see that $u'$ is smooth on the dotted interior segment, we observe that $u_{1}$ extends smoothly by reflection across this segment (since the segment corresponds to an edge in $M$ where $u_1$ satisfies Neumann conditions) and, similarly, $u_{2}$ smoothly extends by negative reflection across this segment (which corresponds to an edge of $M$ where $u_2$ satisfies Dirichlet conditions).  It is then straightforward to verify that $u'$ is an eigenfunction with eigenvalue $\lambda$ for the mixed Robin-Neumann-Dirichlet problem.    The transplantation map $T$ is invertible and isospectrality follows.     The same transplantation map yields the mixed Robin-Neumann-Dirichlet isospectrality of $P$ and $P'$.    

To prove the Steklov-Neumann-Dirichlet isospectrality of $M$ and $M'$ and of $P$ and $P'$, one uses the same expression for the transplantation map $T$, but now acting on Steklov-Neumann-Dirichlet eigenfunctions.  Alternatively, the isospectrality is immediate from the duality between the Steklov-Neumann-Dirichlet and the Robin-Neumann-Dirichlet problem.
\noindent 
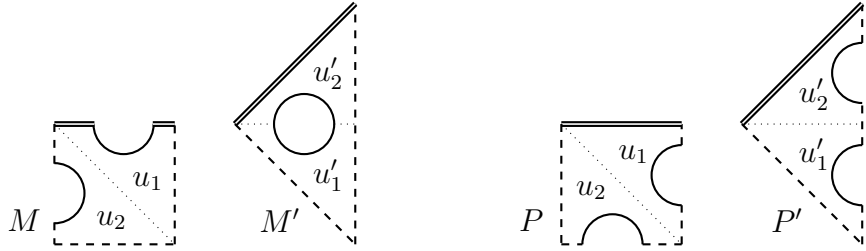
\begin{figure}
\noindent \begin{centering}
\begin{tikzpicture}[scale=0.4,thick]

 \draw (3.1,2.1) node {$u_{1}$};
 \draw (1.9,0.8) node {$u_{2}$};
 \draw (9.1,2.3) node {$\mate{u}_{1}$};
 \draw (9.1,5.7) node {$\mate{u}_{2}$};
 \draw (-1.0,0.8) node {$\manifold$};
 \draw (7.5,0.8) node {$\mate{\manifold}$};

 \draw (9.3,4) arc (0:360:1);
 \draw[dashed] (10, 0) -- (10, 8);
 \draw[double] (10, 8) -- ( 6, 4);
 \draw[dashed] ( 6, 4) -- (10, 0);
 \draw[dotted,thin] (  6,4) -- (7.3,4);
 \draw[dotted,thin] (9.3,4) -- (10,4);

 \draw (1.3,4) arc (180:360:1);
 \draw ( 0,2.7) arc (90:-90:1);

 \draw[dashed] ( 0, 0)    -- (4, 0);
 \draw[dashed] ( 4, 0)    -- (4, 4);
 \draw[double] ( 4, 4)    -- (3.27, 4);
 \draw[double] (1.33, 4) -- (0, 4);
 \draw[dashed] ( 0, 4)    -- (0,2.67);
 \draw[dashed] ( 0,0.73)  -- (0, 0);
 \draw[dotted,thin] ( 0, 4) -- ( 4, 0);

\end{tikzpicture}\hspace{2cm}\begin{tikzpicture}[scale=0.4,thick]

 \draw (2.4,3.0) node {$u_{1}$};
 \draw (1.0,1.8) node {$u_{2}$};
 \draw (8.4,3.0) node {$\mate{u}_{1}$};
 \draw (8.4,5.0) node {$\mate{u}_{2}$};
 \draw (-1 ,0.8) node {$P$};
 \draw (7.5,0.8) node {$\mate{P}$};

 \draw (10,3.3) arc (90:270:1);
 \draw (10,6.7) arc (90:270:1);

 \draw[dashed] (10,  0) -- (10,1.3);
 \draw[dashed] (10,3.3) -- (10,4.7);
 \draw[dashed] (10,6.7) -- (10,  8);
 \draw[double] (10, 8) -- ( 6, 4);
 \draw[dashed] ( 6, 4) -- (10, 0);
 \draw[dotted,thin] (6,4) -- (10,4);

 \draw (2.7,0) arc (0:180:1);
 \draw (4,3.3) arc (90:270:1);

 \draw[dashed] (  0, 0) -- (0.7,0);
 \draw[dashed] (2.7, 0) -- (4, 0);
 \draw[dashed] ( 4,  0) -- (4,1.3);
 \draw[dashed] ( 4,3.3) -- (4, 4);
 \draw[double] ( 4,  4) -- (0, 4);
 \draw[dashed] ( 0,  4) -- (0, 0);
 \draw[dotted,thin] ( 0, 4) -- ( 4, 0);

\end{tikzpicture}
\par\end{centering}

\protect\caption{Robin-Steklov isospectral planar pairs. They are based on the main
example in~\cite{LPP}. Isospectrality
follows from the transplantation~(\ref{eq:Transplantation_in_matrix_form}).\label{fig:RS_isospectral_square_and_triangle}}
\end{figure}

\subsection{Examples using the torus action method}

The torus action method, e.g., Theorem~\ref{thm.schueth}, has led to numerous pairs and families of Dirichlet and Neumann isospectral manifolds as well as isospectral closed manifolds.  All known examples satisfy the additional condition \eqref{F_WBoundaryVolPres} of Theorem~\ref{thm:Torus_method} and therefore have isospectral Dirichlet-to-Neumann operators at all frequencies.
In fact, Proposition~\ref{prop:Boundary_volume_form_preservation} applies to all of them, yielding condition~(2)
in Theorem~\ref{thm:Torus_method}. Letting $B^ n$ and $T^{n}$ denote the $n$-dimensional ball and torus, respectively,
the examples include:
\begin{enumerate}
\item Continuous families of nonisometric metrics on $B^{n}$ for $n\geq8$~\cite{Gordon2001,Schueth2001},
and pairs of such metrics on $B^6$ and $B^7$ \cite{Schueth2001}.
These metrics can be chosen as Euclidean outside of a smaller concentric ball~\cite{Schueth2001}.  \label{enu:Continuous_families_of_isospectral_balls}
\item Continuous families of metrics on $B^n\times\torus^{k}$ for ${n\geq5}$ and $k\geq 2$
that are the restrictions of locally nonisometric homogeneous metrics on $\mathbb{R}^{n}\times\torus^{k}$~\cite{GordonWilson1997}. 
\item For $n\geq6$, if one removes a concentric ball from $B^{n}$ to obtain an annulus $\manifold$ and takes $\rho\equiv 1$ on one of the boundary spheres and $\rho\equiv 0$ on the other, then the metrics in (1) and (2)
restrict to metrics with isospectral sloshing problems on $\manifold$.
\end{enumerate}

\section{Isospectral density functions}\label{sec:density}
 In \cite{Brooks}, R. Brooks modified Sunada's theorem in order to construct isospectral potentials for the Schr\"odinger operator.  Shortly thereafter, a similar method was used to construct isospectral conformally equivalent Riemannian metrics \cite{BPY}.  The technique became standard and produced many new examples.  Later D. Schueth \cite{Schueth2001a} analogously modified the torus action method in order to produce isospectral potentials and isospectral conformally equivalent Riemannian metrics.  In this section, we observe that similar modifications of Theorem~\ref{thm:Sunada} and Theorem~\ref{thm:Torus_method} allow us to produce isospectral boundary density functions for the Steklov spectrum.  Here we carry out the modification of Theorem~\ref{thm:Sunada}; the modification of Theorem~\ref{thm:Torus_method} is similar.

\begin{thm}\label{thm:SunadaDensity}  Let $M$, $G$, $H$, $H'$, $g$ and $\rho$ satisfy the hypotheses of Theorem~\ref{thm:Sunada}.   Assume in addition that there exists an isometry $\tau$ of $(M,g)$, not in $G$, such that $\tau H\tau^{-1}=H'$.  Then for all $\alpha$ not in the Dirichlet spectrum of $(H\bs M,g)$, we have 
$$\Stek_\alpha(H\bs M,g,\rho)=\Stek_\alpha(H\bs M,g,\tau^*\rho)$$
where we continue to denote by $\rho$ and $\tau^*\rho$ the boundary density functions on $H\bs M$ induced by those on $M$.
\end{thm}
\begin{proof}  By Theorem~\ref{thm:Sunada}, $\Stek_\alpha(H\bs M,g,\rho)=\Stek_\alpha(H'\bs M,g,\rho)$ for all $\alpha$ not in the Dirichlet spectrum of $H\bs M$.   By the additional hypothesis of Theorem~\ref{thm:SunadaDensity}, $\tau$ induces an isometry $\tau: (H\bs M,g)\to (H'\bs M,g)$, so we have $\Stek_\alpha(H'\bs M,\rho)=\Stek_\alpha(H\bs M, \tau^*\rho)$.    

\end{proof}
\begin{exa}[Flat surfaces and planar domains]
 In Example 5.6 in \cite{GordonWebbWolpert1992}, the tile in Figure~\ref{fig:Basic_tile} is replaced by a tile $T$ that has not only a reflection symmetry $\beta_0$ as in \ref{exa.domains} but also a rotational symmetry $\tau_0$ that commutes with $\beta_0$.  The tile is pictured in Figure 15 
of \cite{GordonWebbWolpert1992}.  Construct $M=M(G,S)$ and $M_i=M(H_i\bs M,S)$, $i=1,2$ exactly as in \ref{exa.buser} but using the more symmetric tile.    The isometry $\tau_0$ of the basic tile lifts to an isometry $\tau$ of $M$.   The isometry $\tau$ normalizes the group $G$ and  $\tau A\tau^{-1}=(A^t)^{-1}$ for all $A\in G$.   In particular, $\tau H_1\tau^{-1}=H_2$.   Define $\partial_0 T$ as in \ref{exa.buser} and let $\rho_0:\partial_0 T\to\R$ be a boundary density function that is \emph{not} invariant under the restriction to $\partial_0 T$ of the rotational symmetry $\tau_0$.  Denote by $\rho$ the resulting boundary density on $M$.  Then the hypotheses of Theorem~\ref{thm:SunadaDensity} are satisfied with $H_1$ and $H_2$ playing the roles of $H$ and $H'$.   Thus we have $\Stek_\alpha(M_1,\rho)=\Stek_\alpha(M_1, \tau^*\rho)$ for all $\alpha$ not in the Dirichlet spectrum of $M_1$.  
 
Next we construct planar domains.   In the construction in the previous paragraph, impose the additional requirement that $\rho_0$ be invariant under the reflection symmetry $\beta_0$.   As in \ref{exa.domains}, the symmetry $\beta_0$ of the new basic tile lifts to isometric involutions of $M$, $M_1$, and $M_2$.  Let $\mathcal{O}_i$ be the orbifold quotient of $M_i$ by the involution $\beta_i$.  As before, the underlying space of $\mathcal{O}_i$ is a planar domain $D_i$ whose boundary consists of the projection to $\mathcal{O}_i$ of the boundary of $M_i$ (this part is the boundary of the orbifold) together with the a collection of straight line segments corresponding to the singular set of the orbifold.   The boundary density $\rho$ on $M_i$ projects to a density function, still denoted $\rho$, on the first part of the bounary of $D_i$; we extend $\rho$ to the full boundary by setting it to be zero on the orbifold singular set.  Because $\beta_0$ and $\tau_0$ commute, the isometry $\tau:M_1\to M_2$ satisfies $\tau\circ\beta_1=\beta_2\circ\tau$, and thus $\tau$ induces an isometry between the planar domains $D_1$ to $D_2$.  We then have $\Stek_\alpha(D_1,\rho)=\Stek_\alpha(D_1,\tau^*\rho)$ for all $\alpha$ not in the Dirichlet spectrum of $D_1$.    
\end{exa}

The modification of the torus action method is similar.
 
\bibliographystyle{amsalpha}
\bibliography{Robin_Steklov-3}

\newcommand{\etalchar}[1]{$^{#1}$}
\def\cprime{$'$}
\providecommand{\bysame}{\leavevmode\hbox to3em{\hrulefill}\thinspace}
\providecommand{\MR}{\relax\ifhmode\unskip\space\fi MR }
\providecommand{\MRhref}[2]{%
  \href{http://www.ams.org/mathscinet-getitem?mr=#1}{#2}
}
\providecommand{\href}[2]{#2}
\begin{thebibliography}{AMDG{\etalchar{+}}17}

\bibitem[AM12]{ArendtMazzeo2012}
Wolfgang Arendt and Rafe Mazzeo, \emph{Friedlander's eigenvalue inequalities
  and the {D}irichlet-to-{N}eumann semigroup}, Commun. Pure Appl. Anal.
  \textbf{11} (2012), no.~6, 2201--2212.

\bibitem[AMDG{\etalchar{+}}17]{ADGHRS}
T.~Arias-Marco, E.~Dryden, C.~Gordon, A.~Hassannazhad, A.~Ray, and E.~Stanhope,
  \emph{Spectral geometry of the {S}teklov problem on orbifolds}, Int. Math.
  Res. Not. IMRN (2017).

\bibitem[B{\'e}r92]{Berard1992}
Pierre B{\'e}rard, \emph{Transplantation et isospectralit\'e. {I}}, Math. Ann.
  \textbf{292} (1992), no.~3, 547--559.

\bibitem[BGG98]{BGG}
Robert Brooks, Ruth Gornet, and William~H. Gustafson, \emph{Mutually
  isospectral {R}iemann surfaces}, Adv. Math. \textbf{138} (1998), no.~2,
  306--322. \MR{1645582}

\bibitem[BGP00]{BrooksGornetPerry}
Robert Brooks, Ruth Gornet, and Peter Perry, \emph{Isoscattering {S}chottky
  manifolds}, Geom. Funct. Anal. \textbf{10} (2000), no.~2, 307--326.
  \MR{1771427}

\bibitem[BPY89]{BPY}
Robert Brooks, Peter Perry, and Paul Yang, \emph{Isospectral sets of
  conformally equivalent metrics}, Duke Math. J. \textbf{58} (1989), no.~1,
  131--150. \MR{1016417}

\bibitem[Bre72]{Bredon}
Glen~E. Bredon, \emph{Introduction to compact transformation groups}, Academic
  Press, New York-London, 1972, Pure and Applied Mathematics, Vol. 46.
  \MR{0413144}

\bibitem[Bro87]{Brooks}
Robert Brooks, \emph{On manifolds of negative curvature with isospectral
  potentials}, Topology \textbf{26} (1987), no.~1, 63--66. \MR{880508}

\bibitem[Bro01]{Brock2001}
F.~Brock, \emph{An isoperimetric inequality for eigenvalues of the {S}tekloff
  problem}, ZAMM Z. Angew. Math. Mech. \textbf{81} (2001), no.~1, 69--71.

\bibitem[BS08]{BelSh}
Mikhail Belishev and Vladimir Sharafutdinov, \emph{Dirichlet to {N}eumann
  operator on differential forms}, Bull. Sci. Math. \textbf{132} (2008), no.~2,
  128--145. \MR{2387822}

\bibitem[Bus86]{Buser1986}
Peter Buser, \emph{Isospectral {R}iemann surfaces}, Ann. Inst. Fourier
  (Grenoble) \textbf{36} (1986), no.~2, 167--192.

\bibitem[Bus88]{Buser1988}
\bysame, \emph{Cayley graphs and planar isospectral domains}, Geometry and
  analysis on manifolds ({K}atata/{K}yoto, 1987), Lecture Notes in Math., vol.
  1339, Springer, Berlin, 1988, pp.~64--77.

\bibitem[Bus10]{BuserBook}
\bysame, \emph{Geometry and spectra of compact {R}iemann surfaces}, Modern
  Birkh\"auser Classics, Birkh\"auser Boston, Inc., Boston, MA, 2010, Reprint
  of the 1992 edition. \MR{2742784}

\bibitem[BW95]{BerardWebb1995}
Pierre B\'erard and David Webb, \emph{On ne peut pas entendre l'orientabilit\'e
  d'une surface}, C. R. Acad. Sci. Paris S\'er. I Math. \textbf{320} (1995),
  no.~5, 533--536. \MR{1322332}

\bibitem[CESG11]{ColboisElGirouard2011}
Bruno Colbois, Ahmad El~Soufi, and Alexandre Girouard, \emph{Isoperimetric
  control of the {S}teklov spectrum}, J. Funct. Anal. \textbf{261} (2011),
  no.~5, 1384--1399.

\bibitem[CG18]{CianciGir}
Donato Cianci and Alexandre Girouard, \emph{Large spectral gaps for steklov
  eigenvalues under volume constraints and under localized conformal
  deformations}, arXiv:1801.07836, 2018.

\bibitem[CGG17]{ColGitGir}
Bruno Colbois, Alexandre Girouard, and Katie Gittens, \emph{Steklov eigenvalues
  of submanifolds with prescribed boundary in euclidean space},
  arXiv:1711.06458, 2017.

\bibitem[Fri91]{Friedlander1991}
Leonid Friedlander, \emph{Some inequalities between {D}irichlet and {N}eumann
  eigenvalues}, Arch. Rational Mech. Anal. \textbf{116} (1991), no.~2,
  153--160.

\bibitem[FS16]{FraserSchoen2016}
Ailana Fraser and Richard Schoen, \emph{Sharp eigenvalue bounds and minimal
  surfaces in the ball}, Invent. Math. \textbf{203} (2016), no.~3, 823--890.
  \MR{3461367}

\bibitem[GMW05]{GMW}
Carolyn Gordon, Eran Makover, and David Webb, \emph{Transplantation and
  {J}acobians of {S}unada isospectral {R}iemann surfaces}, Adv. Math.
  \textbf{197} (2005), no.~1, 86--119. \MR{2166178}

\bibitem[Gor94]{Gordon1994}
Carolyn~S. Gordon, \emph{Isospectral closed {R}iemannian manifolds which are
  not locally isometric. {II}}, Geometry of the spectrum ({S}eattle, {WA},
  1993), Contemp. Math., vol. 173, Amer. Math. Soc., Providence, RI, 1994,
  pp.~121--131.

\bibitem[Gor01]{Gordon2001}
\bysame, \emph{Isospectral deformations of metrics on spheres}, Invent. Math.
  \textbf{145} (2001), no.~2, 317--331.

\bibitem[Gor09]{G:SunadaTwoDecades}
Carolyn Gordon, \emph{Sunada's isospectrality technique: two decades later},
  Spectral analysis in geometry and number theory, Contemp. Math., vol. 484,
  Amer. Math. Soc., Providence, RI, 2009, pp.~45--58. \MR{1500137}

\bibitem[GP12]{GirouardPolterovich2012}
Alexandre Girouard and Iosif Polterovich, \emph{Upper bounds for {S}teklov
  eigenvalues on surfaces}, Electron. Res. Announc. Math. Sci. \textbf{19}
  (2012), 77--85.

\bibitem[GP17]{GP}
\bysame, \emph{Spectral geometry of the {S}teklov problem (survey article)}, J.
  Spectr. Theory \textbf{7} (2017), no.~2, 321--359. \MR{3662010}

\bibitem[GPPS14]{GirouardParnovskiPolterovichSher2014}
Alexandre Girouard, Leonid Parnovski, Iosif Polterovich, and David~A. Sher,
  \emph{The {S}teklov spectrum of surfaces: asymptotics and invariants},
  Mathematical Proceedings of the Cambridge Philosophical Society \textbf{157}
  (2014), 379--389.

\bibitem[GW97]{GordonWilson1997}
Carolyn~S. Gordon and Edward~N. Wilson, \emph{Continuous families of
  isospectral {R}iemannian metrics which are not locally isometric}, J.
  Differential Geom. \textbf{47} (1997), no.~3, 504--529.

\bibitem[GWW92]{GordonWebbWolpert1992}
C.~Gordon, D.~Webb, and S.~Wolpert, \emph{Isospectral plane domains and
  surfaces via {R}iemannian orbifolds}, Invent. Math. \textbf{110} (1992),
  no.~1, 1--22.

\bibitem[Her11]{Herbrich2011}
Peter Herbrich, \emph{On inaudible properties of broken drums --
  {I}sospectrality with mixed {D}irichlet-{N}eumann boundary conditions},
  arXiv:1111.6789.

\bibitem[Jam14]{Jammes2014}
Pierre Jammes, \emph{Prescription du spectre de {S}teklov dans une classe
  conforme}, Anal. PDE \textbf{7} (2014), no.~3, 529--549.

\bibitem[Kar]{Karp}
Mikhail Karpukhin, \emph{Steklov problem on differential forms},
  arXiv:1705.08951.

\bibitem[Kar17]{Karp2017}
\bysame, \emph{Bounds between {L}aplace and {S}teklov eigenvalues on
  nonnegatively curved manifolds}, Electron. Res. Announc. Math. Sci.
  \textbf{24} (2017), 100--109. \MR{3699063}

\bibitem[KKK{\etalchar{+}}14]{Notices}
Nikolay Kuznetsov, Tadeusz Kulczycki, Mateusz Kwa{\'s}nicki, Alexander Nazarov,
  Sergey Poborchi, Iosif Polterovich, and Bart{\l}omiej Siudeja, \emph{The
  legacy of {V}ladimir {A}ndreevich {S}teklov}, Notices Amer. Math. Soc.
  \textbf{61} (2014), no.~1, 9--22. \MR{3137253}

\bibitem[LP15]{LambertiProvenzano}
Pier~Domenico Lamberti and Luigi Provenzano, \emph{Viewing the {S}teklov
  eigenvalues of the {L}aplace operator as critical {N}eumann eigenvalues},
  Current trends in analysis and its applications, Trends Math.,
  Birkh\"auser/Springer, Cham, 2015, pp.~171--178. \MR{3496508}

\bibitem[LPP06]{LPP}
Michael Levitin, Leonid Parnovski, and Iosif Polterovich, \emph{Isospectral
  domains with mixed boundary conditions}, J. Phys. A \textbf{39} (2006),
  no.~9, 2073--2082. \MR{2211977}

\bibitem[PB10]{BP}
Ori Parzanchevski and Ram Band, \emph{Linear representations and isospectrality
  with boundary conditions}, J. Geom. Anal. \textbf{20} (2010), no.~2,
  439--471. \MR{2579517}

\bibitem[Pes94]{Pesce1994}
Hubert Pesce, \emph{Vari\'et\'es isospectrales et repr\'esentations de
  groupes}, Geometry of the spectrum ({S}eattle, {WA}, 1993), Contemp. Math.,
  vol. 173, Amer. Math. Soc., Providence, RI, 1994, pp.~231--240.

\bibitem[PS15]{PolterovichSher2015}
Iosif Polterovich and David~A. Sher, \emph{Heat invariants of the {S}teklov
  problem}, J. Geom. Anal. \textbf{25} (2015), no.~2, 924--950.

\bibitem[RS12]{RS}
S.~Raulot and A.~Savo, \emph{On the first eigenvalue of the
  {D}irichlet-to-{N}eumann operator on forms}, J. Funct. Anal. \textbf{262}
  (2012), no.~3, 889--914. \MR{2863852}

\bibitem[Sch01a]{Schueth2001a}
Dorothee Schueth, \emph{Isospectral manifolds with different local geometries},
  J. Reine Angew. Math. \textbf{534} (2001), 41--94.

\bibitem[Sch01b]{Schueth2001}
\bysame, \emph{Isospectral metrics on five-dimensional spheres}, J.
  Differential Geom. \textbf{58} (2001), no.~1, 87--111.

\bibitem[Sun85]{Sunada1985}
Toshikazu Sunada, \emph{Riemannian coverings and isospectral manifolds}, Ann.
  of Math. (2) \textbf{121} (1985), no.~1, 169--186.

\bibitem[YY17]{YY}
Liangwei Yang and Chengjie Yu, \emph{Estimates for higher {S}teklov
  eigenvalues}, J. Math. Phys. \textbf{58} (2017), no.~2, 021504, 9.
  \MR{3614611}

\bibitem[Zel92]{Zelditch}
Steven Zelditch, \emph{Isospectrality in the {FIO} category}, J. Differential
  Geom. \textbf{35} (1992), no.~3, 689--710. \MR{1163455}

\end{thebibliography}

\end{document}